\documentclass[]{article}

\usepackage[latin1]{inputenc}
\usepackage{authblk} 
\usepackage{hyperref}
\usepackage{amsmath}
\usepackage{amsfonts}
\usepackage{amssymb}
\usepackage{amsthm}
\usepackage{bbm}
\usepackage{graphicx}
\usepackage{subcaption}
\usepackage[normalem]{ulem} 
\usepackage[dvipsnames]{xcolor} 
\graphicspath{{./figures/}}

\newtheorem{thm}{Theorem}[section]

\title{Bayesian inversion of a diffusion evolution equation with application to Biology.}

\author[1]{Jean-Charles Croix}
\author[1,2]{Nicolas Durrande}
\author[3]{Mauricio A. \'Alvarez}

\affil[1]{Mines Saint-Etienne, Univ Clermont Auvergne, CNRS, UMR 6158 LIMOS, Institut Henri Fayol, Departement GMI, F-42023 Saint-Etienne, France.}
\affil[2]{Prowler.io, Cambridge, UK.}
\affil[3]{Department of Computer Science, University of Sheffield, UK.}

\begin{document}

\maketitle

\begin{abstract}
    A common task in experimental sciences is to fit mathematical models to real-world measurements to improve understanding of natural phenomenon (reverse-engineering or inverse modeling). When complex dynamical systems are considered, such as partial differential equations, this task may become challenging and ill-posed. In this work, a linear parabolic equation is considered where the objective is to estimate both the differential operator coefficients and the source term at once. The Bayesian methodology for inverse problems provides a form of regularization while quantifying uncertainty as the solution is a probability measure taking in account data. This posterior distribution, which is non-Gaussian and infinite dimensional, is then summarized through a mode and sampled using a state-of-the-art Markov-Chain Monte-Carlo algorithm based on a geometric approach. After a rigorous analysis, this methodology is applied on a dataset of the post-transcriptional regulation of Kni gap gene in the early development of Drosophila Melanogaster where mRNA concentration and both diffusion and depletion rates are inferred from noisy measurement of the protein concentration.
\end{abstract}

\section{Introduction}

The problem of diffusion in a porous media, which is ubiquitous in physics, engineering and biology, is usually represented by the following partial differential equation:
\begin{equation}
\begin{aligned}
\frac{\partial y}{\partial t}(x,t)+\lambda(x,t) y(x,t)-D(x,t)\Delta y(x,t)&=f(x,t),\;\forall(x,t)\in\Omega\times]0,T],\\
y(x,t)&=0,\;\forall (x,t)\in\Omega\times\lbrace t=0\rbrace,\\
y(x,t)&=0,\;\forall (x,t)\in\partial\Omega\times]0,T].\\
\end{aligned}
\label{Eq_PDE}
\end{equation}
where the spatial domain is $\Omega\subset\mathbb{R}^n$ ($n\leq 3$) and the final time is $T\in\mathbb{R}^+$ (other initial and boundary conditions are possible). In real world applications, the quantity of interest $y$ (hereafter called the solution of equation \ref{Eq_PDE}) is typically the concentration of some chemical and evolves from a null initial state under three distinct mechanisms: a) direct variation in concentration, given by the source $f$, b) diffusion at rate $D$, c) production or depletion at a rate $\lambda$. Different hypotheses on the parameters lead to a well-posedness of this solution, which will be detailed later on. Besides the traditional computation of the solution from the parameters, one can use this model for the determination of an optimal control (e.g. source leading to the minimization of a particular cost functional) or the identification of parameters from partial knowledge of the solution in an inverse setting. This is the problem that will be of interest in this paper. The motivation comes from a a challenging identification problem in Biology where the objective is to infer jointly the decay rate $\lambda$, the diffusion rate $D$ and the source $f$ given a limited number of noisy observations of the solution $y$. Note that given their physical interpretation, the two rates and the solution must be positive so this constrain has to be taken into account in our inference scheme.

This problem has already been solved by different approaches, under distinct sets of hypotheses. In \cite{Becker2013}, the authors use a system of ordinary differential equations instead of equation \ref{Eq_PDE}, and minimizes a discrete version of a least-square type functional, while confidence intervals on parameters are given by bootstrapping. Alternative methods to solve this problem, in a Bayesian setting, is to use Latent Force Models \cite{Alvarez2013,Sarkka2017}. These approaches assume that unknown physical quantities can be modelled with Gaussian Processes \cite{Rasmussen2004}. In particular, if $f$ is Gaussian and if the decay and the diffusion are constant, $y$ is Gaussian as well (linear transformation of a Gaussian process). The two constant parameters $\lambda$ and $D$ can then be estimated through likelihood maximization. The main drawback of this approach is the difficulty to ensure positiveness of the source function $f$ for instance.

In this work, we apply a more general methodology, based on the recent advent of Bayesian Inverse Problems \cite{Stuart2010a} for infinite dimensional spaces. In a sense, it has the advantage of dealing with the ill-posedness while fully integrating the quantification of uncertainties. Moreover, the possibility of taking naturally physical constraints (such as positivity) will be particularly useful. This paper is organized as follows: section \ref{sec_BayesianInv} presents all the mathematical analysis of the forward model (mapping equation parameters to its solution) and the Bayesian methodology applied in our particular setting. Section \ref{sec_MCMCsampling} focuses on a particular Markov chain Monte-Carlo algorithm that is adapted to the problem and robust to discretization. Finally, section \ref{sec_NumericalExperiments} contains all implementation details and the numerical results obtained on a dataset associated to the developmental biology of the Drosophila Melanogaster.

\section{Bayesian inversion}
\label{sec_BayesianInv}

As previously announced, our goal in this work is to infer a source term $f$ (mRNA concentration) jointly with rates of diffusion $D$ and decay $\lambda$ (i.e. the parameter $u=(\lambda,D,f)$) from noisy and partial measurements of the solution $y$ (gap protein concentration). This problem is ill-posed for multiple reasons: a) the parameter $u$ is infinite dimensional and only finite data are available, b) the solution map is not injective and c) observations are noisy. The typical approach to alleviate this issue is to regularize the problem, usually adding constraints with Tikhonov-Philips functionals, to ensure uniqueness and continuity w.r.t. the observations \cite{Isakov2017, Schuster2012}. Doing so, the regularized problem's solution will be compatible with the dataset, but possibly very different from the \emph{real} parameter (if there is such thing). Additionally, a particularly valuable information is a representation of all parameters $u$ that would lead to similar data, giving precise statement on how the dataset is informative \cite{Ghanem2017, Biegler2011, Sullivan2015}. One approach consists in treating these 2 objectives sequentially, first regularizing then quantifying uncertainty. However, the Bayesian methodology for inverse problems (\cite{Stuart2010a} and more recently \cite{Dashti2015}) is precisely tailored to complete both tasks at once in an elegant manner. One particularity of these recent contributions is to tackle inverse problems directly in function spaces, postponing discretization at the very end for implementation purposes, leading to algorithms robust to the discretization dimension. Indeed, finite approximations of probability measures may be absolutely continuous while their infinite counterparts are mutually singular. This become particularly problematic in MCMC sampling for instance \cite{Cotter2013}.
\paragraph{}
In essence, instead of searching for one particular parameter value that would solve the regularized problem, this approach considers the probability distribution of the parameter given the data. Namely, given a prior distribution (see subsection \ref{sec_ChoicePrior}) and few technical conditions on the forward problem (parameter to data map, see subsection \ref{sec_ForwardModel}), Bayes theorem applies and exhibits a unique posterior distribution (see subsection \ref{sec_PosteriorDistribution}), which is continuous in the data (w.r.t. Hellinger metric). Finally, one may summarize information from this posterior distribution such as expected value or modes (subsection \ref{sec_MAP}). We detail in the next sections how this methodology can be applied to the problem at hand.

\subsection{Forward model analysis}
\label{sec_ForwardModel}

The first step is to detail precisely the required regularity of the solution map from equation \ref{Eq_PDE}. Using common variational techniques from PDE theory (see \cite{Evans2010} or \cite{Brezis2011}), one can show that this equation has a unique weak solution (see theorem \ref{ThmSolutionOperator} which proof is in the appendix) given $u=(\lambda,D,f)$ in a domain $\mathcal{U}$ that will be precised later on. Moreover, this solution evolves smoothly when the parameter varies. Without loss of generality and keeping in mind the Biological application, the underlying physical domain will be $\Omega=]0,L[$ with $L\in\mathbb{R}^+$.

\begin{thm}
    Let $\mathcal{P}=\mathbb{R}\times]0,+\infty[\times L^2([0,T],H^{-1}(\Omega))$, then for all $u\in\mathcal{P}$, equation \ref{Eq_PDE} has a unique weak solution, defining a map:
	\begin{equation*}
	y:u\in \mathcal{P}\to y(u)\in W(0,T,L^2,H_0^1).
	\end{equation*}
	Moreover, this map has the following properties:
	\begin{enumerate}
		\item it satisfies the following estimate $\forall u\in\mathcal{P}$,
		\begin{equation*}
			\lVert y(u)\rVert_{W(0,T,L^2,H^1_0)}\leq\frac{C}{\sqrt{D}}\lVert f\rVert_{L^2([0,T],H^{-1}(\Omega))}
		\end{equation*}
		with $C>0$ constant independent of $u$,
		\item it is locally Lipschitz, $\forall u\in\mathcal{P}$, $\forall r>0$ such that $\mathcal{B}(u,r)\subset\mathcal{P}$, $\exists L(r)>0$, $\forall (u_1,u_2)\in\mathcal{B}(u,r)$,
		\begin{equation*}
			\lVert y(u_1)-y(u_2)\rVert_{W(0,T,L^2,H^1_0)}\leq L(r)\lVert u_1-u_2\rVert_{\mathcal{P}}
		\end{equation*}
		\item it is twice Fréchet differentiable on $\mathcal{P}$.
	\end{enumerate}
	\label{ThmSolutionOperator}
\end{thm}
The proof is given in the appendix. Let us now justify why the properties given in theorem \ref{ThmSolutionOperator} are important for the Bayesian inversion (most of them used in section \ref{sec_PosteriorDistribution}):
\begin{enumerate}
	\item The energy estimate will be critical to establish the continuity of the posterior w.r.t. data, because it relates sufficient integrability conditions on $y$ w.r.t. $u$,
	\item Continuity (implied by Fréchet differentiability or local Lipschitz behaviour) will be used to show that the solution map is measurable,
	\item Second order Fréchet differentiability will be necessary for geometric methods in optimization (research of modes) and sampling (Markov-Chain Monte-Carlo),
	\item The local Lipschitz behaviour is used in the characterization of posterior modes (Maximum a Posteriori).
\end{enumerate}
In the rest of this work, we will restrict ourselves to the subset
\begin{equation*}
	\mathcal{U}:=\mathbb{R}^+\times]0,+\infty[\times\mathcal{C}([0,T]\times[0,L],\mathbb{R}),
\end{equation*}
which is implicitly equipped with the norm
\begin{equation*}
	\lVert u\rVert_{\mathcal{U}}=\vert\lambda\vert+\vert D\vert+\lVert f\rVert_{\infty}.
\end{equation*}
Since $\mathcal{U}\subset\mathcal{P}$, the solution map is well defined on this subset and keeps all its smoothness properties. Moreover, one can show that a weak solution of equation \ref{Eq_PDE} for $u\in\mathcal{U}$ is also a strong solution \cite{Brezis2011,Evans2010}, but we don't need these regularity results here.

\subsection{Choice of prior distribution}
\label{sec_ChoicePrior}

The second step is to choose a prior probability distribution on $\mathcal{U}$, encoding all knowledge on the physics at hand, while being simple enough to keep analysis tractable. Here are the constraints given by the biological application:
\begin{itemize}
	\item $(\lambda,D)$ must be positive (which imposes decay, not production),
	\item $f$ must be positive and continuous at all time and position (it is a concentration).
\end{itemize}
Starting with the depletion and diffusion parameters $(\lambda, D)$, we choose respectively Borel prior distributions $\mu_0^\lambda$ and $\mu_0^D$ on $\mathbb{R}^+$ with densities w.r.t. Lebesgue's measure. Now, since $f$ must be positive, we re-parametrize the problem with
\begin{equation}
	f^*=\exp(f),
	\label{Eq_Reparametrization}
\end{equation}
where $f\in\mathcal{C}([0,T]\times[0,L],\mathbb{R})$. By selecting a Borel probability measure $\mu_0^f$ on the Banach space $\mathcal{C}([0,T]\times[0,L],\mathbb{R})$, both continuity and positivity of the new source $f^*$ will be ensured almost-surely. In this work, we choose $\mu_0^f$ as a Gaussian distribution with covariance operator $\mathcal{C}$ with continuous realizations (see \cite{Bogachev1998} for a presentation of infinite dimensional Gaussian measures). Finally, we assume independence between the three components, leading to the following prior distribution:
\begin{equation}
\mu_0(du):=\mu_0^\lambda(d\lambda)\otimes\mu_0^D(dD)\otimes\mu_0^f(df).
\label{EqPriorMeasure}
\end{equation}
These choices clearly ensure positivity of $u$ (in the previous sense) and $\mu_0(\mathcal{U})=1$. The exponential map in equation \ref{Eq_Reparametrization} can be replaced with any sufficiently differentiable function from $\mathbb{R}$ to $\mathbb{R}^+$ (to keep the second order Fréchet differentiability of the solution map). Alternative distributions are possible for $f$ (Besov measure from \cite{Dashti2013} or more general convex measures from \cite{Hosseini2017a}) for the regularization. In practice however, our choice is also motivated by the fact that one can find a Gaussian reference measure $\mu_{ref}$ in the form
\begin{equation*}
	\mu_{ref}=\mu_{ref}^\lambda\otimes\mu_{ref}^D\otimes\mu_{ref}^f=
	\mathcal{N}(\lambda_{ref},\sigma_\lambda^2)\otimes \mathcal{N}(D_{ref},\sigma_D^2)\otimes \mathcal{N}(0,\mathcal{C})=\mathcal{N}(u_{ref},\mathcal{C}_{ref})
	\label{Eq_GaussianRef}
\end{equation*}
where $u_{ref}=(\lambda_{ref},D_{ref},0)$ and
\begin{equation*}
	\mathcal{C}_{ref}:u\in\mathbb{R}^2\times L^2([0,L])\to (\sigma^2_\lambda\lambda,\sigma^2_DD,\mathcal{C}f)\in\mathbb{R}^2\times L^2([0,L]),
\end{equation*}
such that $\mu_0<<\mu_{ref}$. Indeed, choose $(\lambda_{ref},D_{ref})\in\mathbb{R}^2$ and $\sigma_\lambda^2,\sigma_D^2>0$ then $\mu_0<<\mu_{ref}$ with
\begin{equation*}
\frac{d\mu_0}{d\mu_{ref}}(u)=\frac{d\mu_0^\lambda}{d\mu_{ref}^\lambda}(\lambda)\frac{d\mu_0^D}{d\mu_{ref}^D}(D).
\label{Eq_RNdensityRef}
\end{equation*}
This reference will be critical for modes analysis (section \ref{sec_MAP}) and MCMC sampling (section \ref{sec_MarkovKernel}).

\subsection{Posterior distribution}
\label{sec_PosteriorDistribution}

The third and last step in the theoretical analysis of Bayesian inversion is to show that this particular setting (forward model and prior distribution) leads to a well defined posterior measure using Bayes theorem. This is the purpose of theorem \ref{ThmPosteriorMeasure} which is a direct application of the theory initially developed in \cite{Stuart2010a}. Consider a dataset $z=(z_i)_{i\in[1,n]}$ which corresponds to observations at different times and locations $(t_i,x_i)_{i\in[1,n]}$ and assume they are produced from the following model (in vector notations):
\begin{equation}
z=\mathcal{G}(u)+\eta,
\label{EqObservationModel}
\end{equation}
where $\eta\sim\mathcal{N}(0,\sigma_\eta^2I_n)$ ($I_n$ being the identity matrix of dimension $n$) and $\mathcal{G}:\mathcal{U}\to\mathbb{R}^n$ is the observation operator, mapping directly the PDE parameter $u$ to the value of the associated solution $y$ at measurement locations (composition of solution map $y$ with Diracs). The following theorem, which is proved in the appendix, establishes the existence of a posterior probability measure $\mu_z$ (the solution of our inverse problem), expressing how observations $z$ changed prior beliefs $\mu_0$ on the parameter $u$.
\begin{thm}
	Let $\mathcal{G}$ be the observation operator defined in equation \ref{EqObservationModel} and $\mu_0$ the probability measure defined in equation \ref{EqPriorMeasure} (satisfying $\mu_0(\mathcal{U})=1$), then there exists a unique posterior measure $\mu_z$ for $u\vert z$. It is characterized by the following Radon-Nikodym density w.r.t. $\mu_0$:
	\begin{equation*}
	\frac{d \mu_z}{d\mu_0}(u)=\frac{1}{Z(z)}\exp\left(-\phi(u;z)\right),
	\end{equation*}
	with
	\begin{equation*}
	\phi(u;z)=\frac{1}{2\sigma_\eta^2}\lVert z-\mathcal{G}(u)\rVert^2,
	\end{equation*}
	and
	\begin{equation*}
	Z(z)=\int_\mathcal{U}\exp(-\phi(u;z)\mu_0(dz).
	\end{equation*}
	Furthermore, the two following integrability conditions:
	\begin{itemize}
		\item $\mathbb{E}^{\mu_0^D}\left[D^{-\frac{1}{2}}\right]<+\infty$,
		\item $\exists\kappa>0,\;\mathbb{E}^{\mu_0^f}\left[\exp(\kappa\lVert f\rVert_\infty)\right]<+\infty$,
	\end{itemize}
	imply the continuity of $\mu_z$ in the data w.r.t. Hellinger distance.
	\label{ThmPosteriorMeasure}
\end{thm}
Theorem \ref{ThmPosteriorMeasure} gives two distinct results: a) the existence and uniqueness of a posterior (as long as $\mu_0$ is Radon and $\mu_0(\mathcal{U})=1$), b) well-posedness of the Bayesian inverse problem under additional integrability conditions of certain functions. In particular, the need for an exponential moment under $\mu_0^f$ comes from the re-parametrization in equation \ref{Eq_Reparametrization}. If one chooses a different map between $f^*$ and $f$, this condition may be considerably relaxed (using a polynomial map for instance).

\subsection{Maximum a posteriori}
\label{sec_MAP}

In the previous section, we proved the well-posedness of the Bayesian inverse problem under specific integrability conditions. However, the posterior distribution is only known up to a multiplicative constant, through its density w.r.t. $\mu_0$. In our application, we will need to summarize $\mu_z$, which is usually done by the selection of elements of interest in $\mathcal{U}$, such as the conditional mean (CM) or its modes. We saw in theorem \ref{ThmPosteriorMeasure} that the posterior expected value is continuous in the data (consequence of Hellinger continuity, see \cite{Dashti2015}) but its optimality properties (under the frequentist paradigm) are not yet well-understood in infinite dimension to the best of our knowledge. This is why posterior modes (or Maximum a Posteriori) are more and more considered instead. Indeed, they provide a clear link with Tikhonov-Philips regularization (see \cite{Dashti2013,Helin2015,Agapiou2017}) and a practical optimization problem (in case of Gaussian or Besov priors) which can be solved numerically, see theorem \ref{ThmMAP} (which proof is given in the appendix).
\begin{thm}
	Let $\mu_0$ be the prior probability measure defined in equation \ref{EqPriorMeasure} and $\mu_{ref}$ the Gaussian reference measure from equation \ref{Eq_RNdensityRef}. Suppose additionally that
	\begin{equation*}
		u\in\mathcal{U}\to ln\left(\frac{d\mu_0}{d\mu_{ref}}(u)\right)\in\mathbb{R}
	\end{equation*}
	is locally Lipschitz, then the modes of $\mu_z$ are exactly the minimizers of the following (generalized) Onsager-Mashlup functional:
	\begin{equation*}
		I(u):=\Phi(u;z)+\frac{1}{2}\lVert u-u_{ref}\rVert^2_{\mu_{ref}}-ln\left(\frac{d\mu_0}{d\mu_{ref}}(u)\right),
	\end{equation*}
	where $\lVert.\rVert_{\mu_{ref}}$ and $u_{ref}$ are respectively the norm of the Cameron-Martin space and the mean of $\mu$. A minimizer will be noted $u_{MAP}=(\lambda_{MAP},D_{MAP},f_{MAP})$.
	\label{ThmMAP}
\end{thm}

The precise application of this theorem to our biological setting is done in section \ref{sec_ChoiceOfPriors}.

\section{Metropolis-Hastings algorithm}
\label{sec_MCMCsampling}

As it was previously announced, our motivation for the Bayesian methodology is the quantification of uncertainty, which will be done by simulation. Among the vast catalogue of methods for probability distributions sampling (Sequential Monte-Carlo, Approximate Bayesian Computations, Transport Maps, etc...), Markov chain Monte-Carlo is very popular (MCMC, see \cite{Brooks2011}) and well defined on function spaces \cite{Tierney1998} even though ergodicity analysis of such algorithm is still in its infancy \cite{Hairer2005,Hairer2007,Hairer2014,Rudolf2018}. After a short presentation of the functional Metropolis-Hastings algorithm (section \ref{sec_MH}), we will focus on a state-of-the-art Markov kernel designed to sample from Gaussian measures (section \ref{sec_MarkovKernel}) and adapt it when the prior is not Gaussian, but absolutely continuous with respect to a Gaussian reference.

\subsection{Metropolis-Hastings on function spaces}
\label{sec_MH}

The Metropolis-Hastings algorithm (MH) is a very general \cite{Tierney1998} method to design Markov chains to sample from a given probability measure. It is based on a two-step process on each iteration:
\begin{enumerate}
    \item Given a current state $u\in\mathcal{U}$, propose a new candidate $v$ according to a proposal Markov kernel $Q(u,dv)$ (it is a probability distribution on $\mathcal{U}$ for almost any $u\in\mathcal{U}$),
	\item Accept the new state $v$ with probability $\alpha(u,v)$ or remain at $u$.
\end{enumerate}
This algorithm provides a sample distributed under a predefined probability measure $\mu$, if one selects $\alpha$ and $Q$ in a specific way (see \cite{Tierney1998} for a discussion in general state spaces). For instance, let $\nu(du,dv)=\mu(du)Q(u,dv)$ and $\nu^T(du,dv)=\mu(dv)Q(v,du)$, the Metropolis-Hastings algorithm typically considers the following acceptance probability:
\begin{equation}
\alpha_{MH}(u,v)=\min\left(1, \frac{d\nu^T}{d\nu}(u,v)\right),
\end{equation}
which, in particular, requires the absolute continuity of $\nu^T$ w.r.t. $\nu$ (detailed balance of the Markov chain). Contrary to finite dimensional situations, this condition may be difficult to satisfy and a common way to overcome this situation in Bayesian Inverse problems (see \cite{Dashti2015,Girolami2011,Beskos2017a,Cotter2013,Hairer2014}) is to select $Q$ revertible w.r.t. $\mu_0$. Indeed, in this case (with $\nu_0(du,dv)=\mu_0(du)Q(u,dv)$):
\begin{equation}
	\frac{d\nu^T}{d\nu}(u,v)=\frac{\frac{d\nu^T}{d\nu^T_0}(u,v)}{\frac{d\nu}{d\nu_0}(u,v)}=\frac{\frac{d\mu_z}{d\mu_0}(v)}{\frac{d\mu_z}{d\mu_0}(u)}=\exp\left(\phi(u;z)-\phi(v;z)\right).
	\label{Eq_Acceptance}
\end{equation}
In theory, the MH algorithm may be implemented with a large family of proposal kernels $Q$. In practice however, they need to be as efficient as possible and thus adapted to the problem at hand. Two common desirable properties for $Q$ are:
\begin{itemize}
	\item Adjust the proposal to locally mimic the target distribution $\mu_z$,
	\item Include a step size to tune acceptance probability.
\end{itemize}
These two properties may be used to trade-off self-correlation, acceptance rates and convergence speed to high interest areas of the parameter space. The next section presents an algorithm with both properties adapted to Gaussian priors which will then be applied to the problem at hand.

\subsection{Geometric MCMC under Gaussian reference}
\label{sec_MarkovKernel}

We are now going to detail a specific Markov proposal kernel $Q$, tailored to sample distributions having a density w.r.t. a Gaussian measure $\mu_{ref}$. Most of recent work on infinite dimensional MCMC methods are based on the following Langevin stochastic differential equation:
\begin{equation}
	\frac{du}{dt}=-\frac{1}{2}K(u)\left(\mathcal{C}_{ref}^{-1}(u-u_{ref})+\nabla_u\phi(u;z)\right)+\sqrt{K(u)}\frac{dW}{dt},
	\label{Eq_Langevin}
\end{equation}
where $K(u)$ is a (possibly position-dependent) preconditioner, $W$ a cylindrical Brownian motion and $\nabla_u\phi(u;y)$ the gradient in $u$ of the negative log-likelihood. According to \cite{Beskos2017a}, a semi-implicit discretization of equation \ref{Eq_Langevin} leads to a Markov chain with the following kernel:
\begin{equation}
	Q(u,dv)=\mathcal{N}\left(\rho (u-u_{ref})+u_{ref}+\sqrt{1-\rho^2}\frac{\sqrt{h}}{2}g(u),K(u)\right),
	\label{Eq_MarkovKernel}
\end{equation}
where $h>0$ is a step-size parameter, $\rho=\frac{1-h}{1+h}$ and:
\begin{equation*}
	g(u)=-K(u)\left[(\mathcal{C}_{ref}^{-1}-K(u)^{-1})(u-u_{ref})+\nabla_u\phi(u;z)\right].
\end{equation*}
This dynamic explores the parameter space with a balance between Newton-type descent to zones of high density and Gaussian exploration. The philosophy behind this kernel is to use alternative Gaussian reference measures locally adapted to the posterior distribution, since it has been recently showed that higher efficiency is obtained from operator weighted proposals (\cite{Law2014} and later generalized in \cite{Beskos2017a} and \cite{Cui2016b}). Indeed, highly informative datasets may result in a posterior measure significantly different from the prior in likelihood-informed directions and non-geometric kernels (such as Independent sampler or preconditioned Crank-Nicholson) are ineffective in this case. However, the infinite dimensional manifold Modified Adjusted Langevin Algorithm ($\infty$-mMALA) considers a specific preconditioner:
\begin{equation*}
	K(u)=\left(\mathcal{C}_{ref}^{-1}+H_\Phi(u)\right)^{-1},
\end{equation*}
where $H_\Phi(u)$ is the Gauss-Newton Hessian matrix of $\phi$, which locally adapts to the posterior. This kernel does not preserve the distribution $\mu_z$ but is shown to be absolutely continuous w.r.t. the reference measure $\mu_{ref}$, almost-surely in $u$ (under technical assumptions regarding $K(u)$) and the Radon-Nikodym density is:
\begin{equation*}
	\frac{dQ(u,dv)}{d\mu_{ref}}(v)=\frac{dN\left(\frac{\sqrt{h}}{2}g(u),K(u)\right)}{dN(0,\mathcal{C})}\left(\frac{v-\rho (u-u_{ref})-u_{ref}}{\sqrt{1-\rho^2}}\right),
\end{equation*}
and noting $w=\frac{v-\rho (u-u_{ref})-u_{ref}}{\sqrt{1-\rho^2}}$ as it is done in \cite{Beskos2017a}, it finally comes:
\begin{equation*}
\begin{split}
\frac{dQ(u,dv)}{d\mu_{ref}}(u,v)&=\exp\left(-\frac{h}{8}\vert K(u)^{-\frac{1}{2}}g(u)\vert^2+\frac{\sqrt{h}}{2}\langle K(u)^{-\frac{1}{2}}g(u),K(u)^{-\frac{1}{2}}w\rangle\right.\\
&\left.-\frac{1}{2}\langle w,H_\Phi(u)w\rangle\right)\left\vert\mathcal{K}^{\frac{1}{2}}K(u)^{-\frac{1}{2}}\right\vert.
\end{split}
\end{equation*}
Finally, the acceptance probability associated to the Markov kernel from equation \ref{Eq_MarkovKernel} is
\begin{equation*}
	\alpha(u,v)=\min\left(1,\frac{\frac{dQ}{d\mu_{ref}}(v,u)\frac{d\mu_z}{d\mu_0}(v)}{\frac{dQ}{d\mu_{ref}}(u,v)\frac{d\mu_z}{d\mu_0}(u)}\right).
\end{equation*}
This algorithm is well-defined on function spaces (reversibility is ensured w.r.t. $\mu_0$), thus it is robust to discretization as required. The $\infty$-mMALA proposal may be computationally expensive, as it requires to compute both gradient $\nabla\Phi$, Gauss-Newton Hessian $H_\Phi$ and the Cholesky decomposition of $K(u)^{-1}$ at each step. However, different dimension reduction techniques can be used (split in \cite{Beskos2017a} or likelihood-informed in \cite{Cui2016b}) to reduce the computational burden. A second alternative is to choose a constant preconditioner, located at a posterior mode for instance (similar to HMALA in \cite{Cui2016b} and gpCN in \cite{Rudolf2018}).

\section{Numerical application}
\label{sec_NumericalExperiments}

We now turn to the practical implementation of the previous methodology on the problem of reverse-engineering for post-transcriptional gap-gene in Drosophila Melanogaster. First, we precise our choice of distributions for the parameters compatible with previous assumptions, and give a random series representation for $f$ and precise the generalized Onsager-Mashlup functional. Then, we provide quantitative results on the dataset taken from \cite{Becker2013}, consisting in protein concentration measurements irregularly spread in space and time.

\subsection{Choice of measures}
\label{sec_ChoiceOfPriors}

We will now specify our choice of prior measure $\mu_0$ with justifications:
\begin{itemize}
	\item $\mu_0^\lambda,\mu_0^D$. Concerning the decay parameter, the only constraint given in the problem so far is positivity and Lebesgue density. However, the diffusion must also satisfy an integrability condition from theorem \ref{ThmPosteriorMeasure}, which is clearly the case for uniform distributions. Finally, we choose $\mu_0^\lambda=\mathcal{U}([0,\lambda_m])$ and $\mu_0^D=\mathcal{U}([0,D_m])$, maximum parameters $\lambda_m$ and $D_m$ being tuned to $0.5$ to be sufficiently large w.r.t. previous estimations from \cite{Becker2013}.
	\item $\mu^f_0$. The prior measure $\mu_0^f$ will be chosen as a centred Gaussian measure on $L^2([0,T]\times[0,L])$, with covariance operator on $L^2([0,T]\times[0,L])$:
	\begin{equation*}
	\mathcal{C}=(-\Delta)^{-\alpha},
	\end{equation*}
	where $\alpha\in\mathbb{R}^+$ is a smoothness parameter tuned to ensure almost-surely continuity of the samples. The precise eigen-decomposition is given as follows $i_1,i_2\geq 1$:
	\begin{eqnarray*}
		\varphi_{i_1,i_2}(x,t)&=&\frac{1}{\sqrt{LT}}\sin\left(i_1\frac{\pi }{L}x\right)\sin\left(i_2\frac{\pi}{T}t\right),\\
		\lambda_{i_1,i_2}&=& \left(\left(\frac{i_1}{\pi L}\right)^2+\left(\frac{i_2}{\pi T}\right)^2\right)^{-\alpha}.
	\end{eqnarray*}
	\item With these choices done, the Radon-Nikodym density of the prior distribution w.r.t. the reference measure is
	\begin{equation*}
		\frac{d\mu_0}{d\mu_{ref}}(u)=\frac{2\pi\sigma_\lambda\sigma_D}{\lambda_mD_m}\exp\left(\frac{(\lambda-\lambda_{ref})^2}{2\sigma_\lambda^2}+\frac{(\lambda-D_{ref})^2}{2\sigma_D^2}\right)\mathbbm{1}_{[0,\lambda_m]}(\lambda)\mathbbm{1}_{[0,D_m]}(D),
	\end{equation*}
\end{itemize}
which is locally Lipschitz (mean-value theorem). We now tune the parameters of $\mu_{ref}$, by simply choosing $\lambda_{ref}=\frac{\lambda}{2}$, $\sigma_\lambda^2=\frac{\lambda_m^2}{12}$, $D_{ref}=\frac{D_m}{2}$ and $\sigma_D^2=\frac{D_m^2}{12}$ (minimizing Kullback-Leibler divergence). Finally, we are capable of specifying the exact form of the generalized Onsager-Mashlup functional using theorem \ref{ThmMAP}:
\begin{equation*}
\begin{split}
I(u)&=\frac{1}{2\sigma_\eta^2}\lVert z-\mathcal{G}(u)\rVert^2+\frac{1}{2}\lVert f\rVert_{\mu_0^f}^2.
\end{split}
\label{eq_gen_ons_mach}
\end{equation*}
In particular, parameters $\lambda$ and $D$ are only influenced by their range and contribution to the likelihood (the uniform prior is non-informative), contrary to the source $f$.

\subsection{Prior and solution map discretization}
\label{sec_Discretization}

The analysis conducted in all previous sections happens to be valid for infinite dimensional quantities. In practice however, one needs to discretize for numerical experiments. In this work, the solution map is approximated using finite elements in space (FEniCS library in Python, see \cite{Alnae2015} and \cite{Langtangen2017}) and finite differences in time. We use 100 basis functions and 30 time steps on a desktop computer (Intel i7-3770 with 8Gb of RAM memory)\footnote{All codes are available online at \url{https://github.com/JeanCharlesCroix/2018_Bayesian_estimation}.}. We set $L=100$ and $T=100$ is the final time. Concerning the prior measure, we use a truncated Karhunen-Loeve basis of $f$ to simulate from it
\begin{equation}
	\tilde{f}=\sum_{1\leq i_1,i_2\leq N}\sqrt{\lambda_{i_1,i_2}}\xi_i\varphi_{i_1,i_2},
\end{equation}
where $(\xi_{i_1,i_2})_{1\leq i_1,i_2\leq N}$ are i.i.d. $\mathcal{N}(0,1)$ random variables. We consider $N=10$ (100 basis functions) thus $\tilde{u}=(\lambda,D,\tilde{f})$ is of dimension 102. All quantities related to negative log-likelihood derivatives (Gradient and Gauss-Newton Hessian matrix) are numerically computed using discrete adjoint methods (see \cite{Hinze2009} or \cite{Heinkenschloss2008}) to keep scalability in $N$. The initial point in the chain is chosen at the MAP location, obtained by minimization of the functional in equation \ref{eq_gen_ons_mach} (Prior based initialization results in long burnin phase). Practical optimization is done using L-BFGS-B algorithm from the Scipy library \cite{Byrd1995}.

\subsection{Results}
\label{sec_Melanogaster}

We now turn to our main objective, the inversion and uncertainty quantification of gap-gene protein concentration from \cite{Becker2012}. The dataset consists of 508 different measures which are non-uniformly spread in time and space (precise repartition can be seen in figure \ref{fig_MAPpost}). The noise variance parameter is estimated using the following routine (10 iterations, 3 multi-start each):
\begin{enumerate}
	\item Find $u_{MAP}$ minimizing $I$ using the current noise level $\sigma_\eta^2$.
	\item Update the current variance estimation $\sigma_\eta^2=\frac{2}{n-1}\lVert z-\mathcal{G}(u_{MAP})\rVert^2$.
	\item Go back to $1$.
\end{enumerate}
The resulting estimated noise level is $\sigma_\eta^2=11,77$. With this estimated value, we compute our initial MAP estimate (numerical minimization of the Onsager-Machlup functional) and use it as initial point in the MCMC sampling. The Markov chain is ran for $11000$ total iterations and the resulting traceplot is given in figure \ref{fig_traceplot} for negative log-likelihood, decay, diffusion and first three components of $f$.
\begin{figure}[h!]
	\centering
	\includegraphics[width=1.0\linewidth]{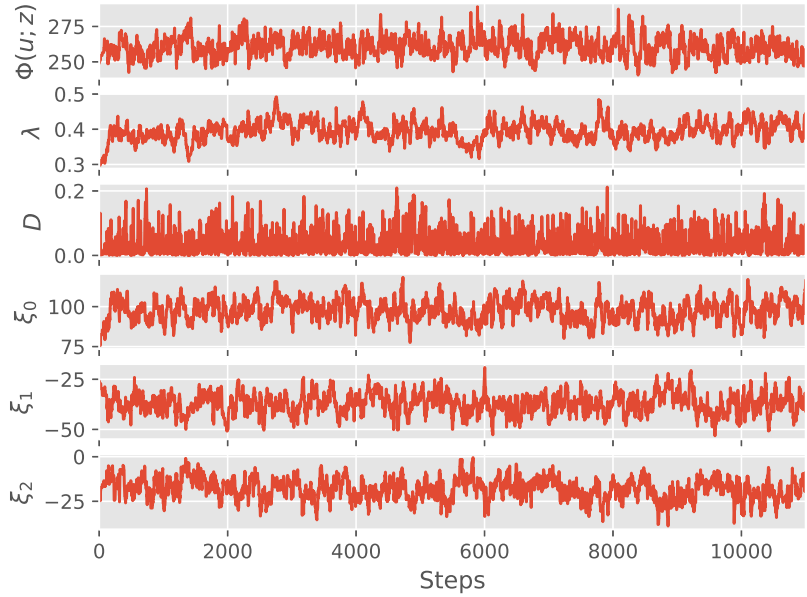}
	\caption{Trace plots of $\Phi(u;z)$, $\lambda$, $D$, $\xi_0$, $\xi_1$ and $\xi_2$ (the first 1000 iterations are burned).}
	\label{fig_traceplot}
\end{figure}
The first thousand iterations are used as burnin and according to the autocorrelation function (figure \ref{fig_acf}), we choose to keep one iteration out of a hundred as posterior sample (thinning).
\begin{figure}[h!]
	\centering
	\includegraphics[width=1.0\linewidth]{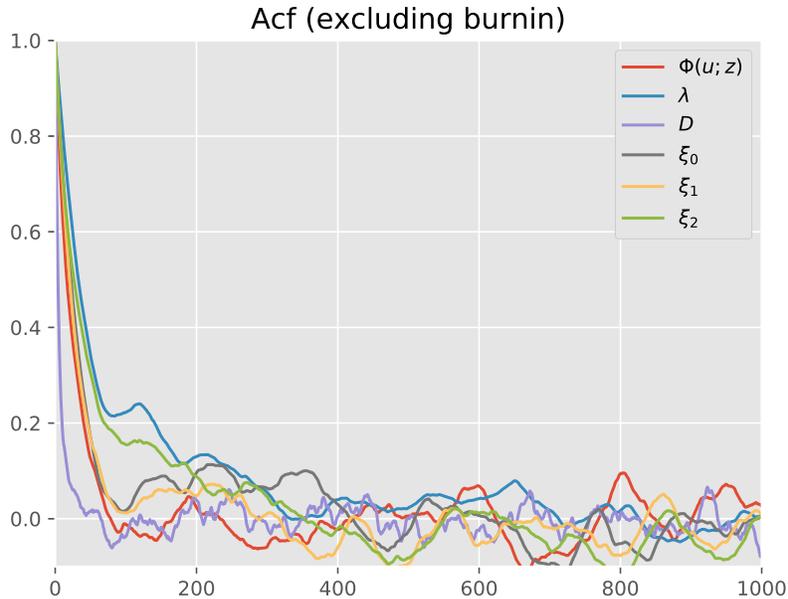}
	\caption{Autocorrelation of negative of $\Phi(u;z)$, $\lambda$, $D$, $\xi_0$, $\xi_1$ and $\xi_2$, excluding burnin sample.}
	\label{fig_acf}
\end{figure}
From this, we compute both posterior mean and MAP estimates, the precise values of decay, diffusion, negative log-likelihood and Onsager-Machlup functional being given in table \ref{table_estimates}.
\begin{table}[h!]
	\centering
	\begin{tabular}{c|c|c|c|c}
		Parameter & $\lambda$ & $D$ & $\Phi(u;z)$ & $I(u)$ \\
		\hline
		MAP (Onsager-Machlup) & $3.00*10^{-1}$ & $1.0*10^{-8}$ & $248.91$ & $440.51$ \\ 
		MAP (MCMC) & $4.25*10^{-1}$ & $3.23*10^{-2}$ & $246.72$ & $460.73$ \\
		Conditional mean (MCMC) & $3.97*10^{-1}$ & $3.30*10^{-2}$ & $243.54$ & $422.41$\\
	\end{tabular}
	\caption{Values of decay, diffusion, negative log-likelihood and Onsager-Machlup functional for different estimators.}
	\label{table_estimates}
\end{table}
Additionally to the estimated values, one can also look at the marginal distribution on figure \ref{fig_scatter}.
\begin{figure}[h!]
	\centering
	\includegraphics[width=1\linewidth]{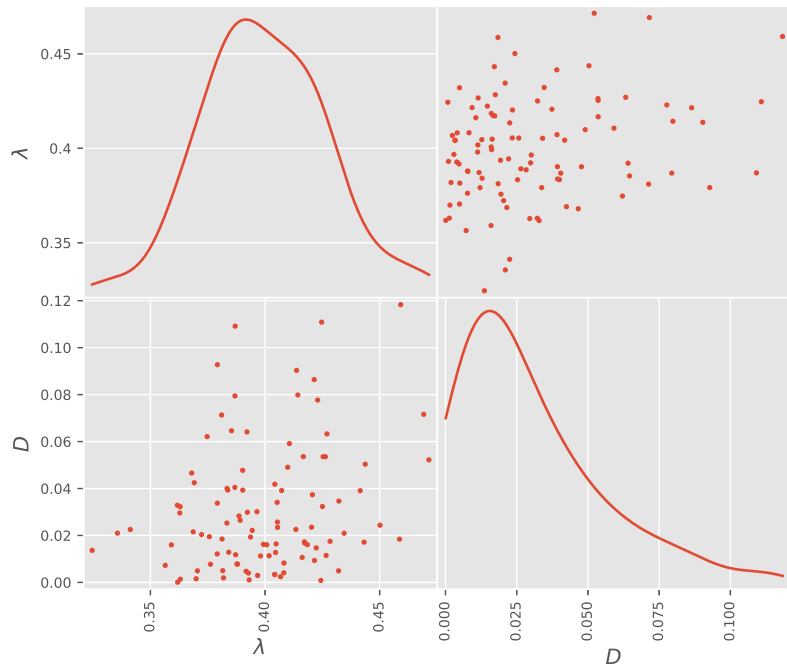}
	\caption{Marginal posterior distributions for $\lambda\vert z$ and $D\vert z$.}
	\label{fig_scatter}
\end{figure}
Concerning the MAP estimator (figure \ref{fig_MAPpost}), we recover both events described in \cite{Becker2012} and \cite{Becker2013}, that is 2 pikes of protein concentrations.
\begin{figure}[h!]
	\centering
	\includegraphics[width=1.0\linewidth]{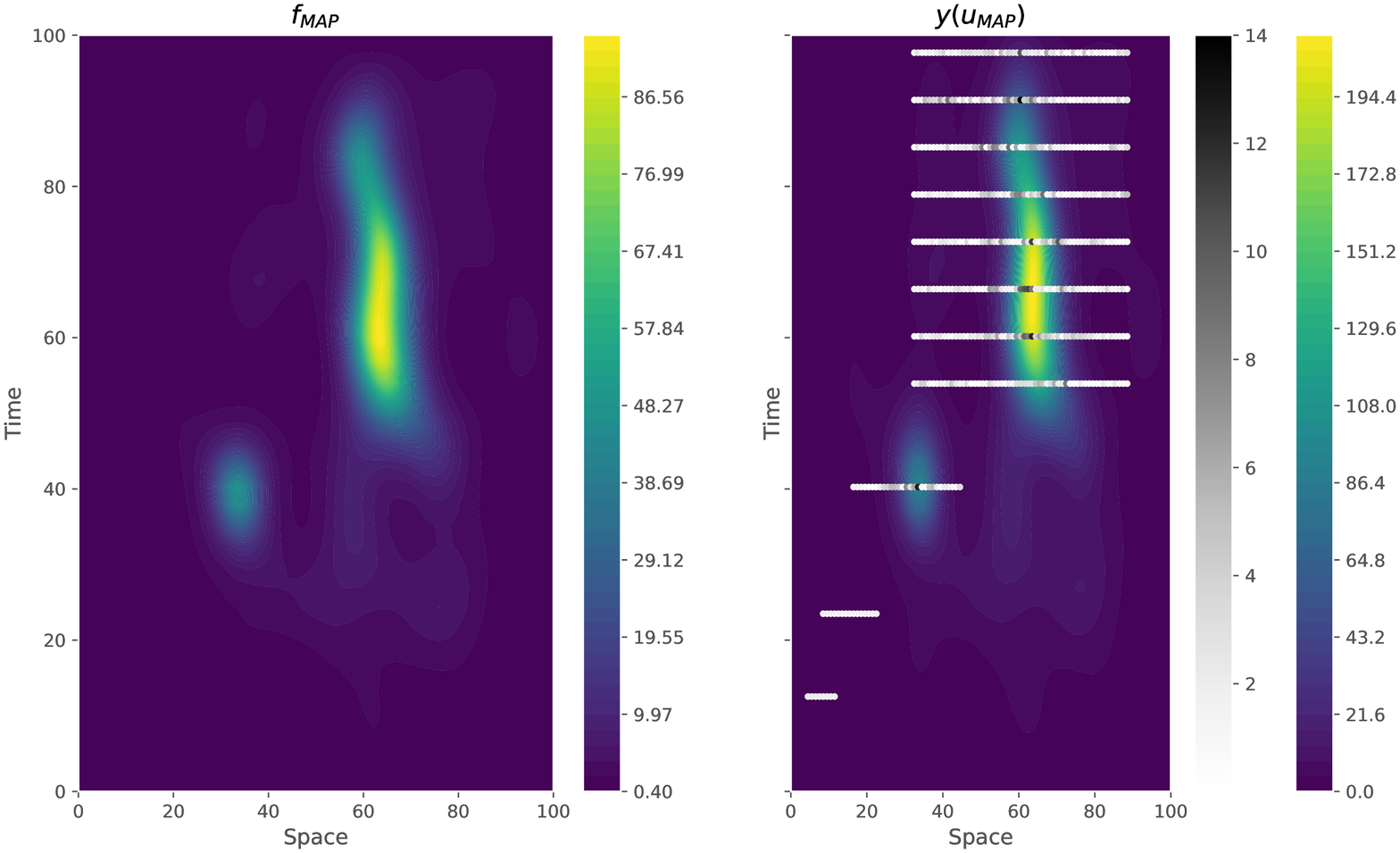}
	\caption{MAP estimator from MCMC sampling. Left: Estimated source term $f_{MAP}$, right: estimated solution $y(u_{MAP})$ with absolute error at data locations. Grey bar represents the error level between data and $y(u_{MAP})$.}
	\label{fig_MAPpost}
\end{figure}
The first happens on the anterior part of the embryo in the early experiment $(x=35,t=35)$. The second is much more intense and happens in the posterior part during the second half of the experiment. The estimated source explains these with an intense and localized increase in concentration. Finally, the uncertainty on both source and solution around data seems to be really low, which provides a good confidence on the level of mRNA at this time and part of the embryo (see figure \ref{fig_var}). However, the point-wise variance on the solution $y$ remains important before the first observations.
\begin{figure}[h!]
	\centering
	\includegraphics[width=1.0\linewidth]{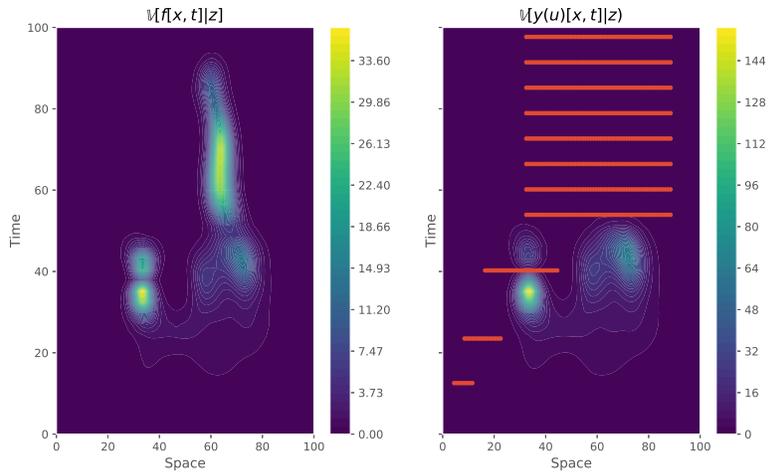}
    \caption{Posterior point-wise variance of source (left figure) and solution (right figure).}
	\label{fig_var}
\end{figure}

\section{Conclusion}
In this work, we applied the Bayesian inverse problem methodology from \cite{Stuart2010a} to a practical Biological dynamical system. Doing so, we provide a rigorous and detailed analysis of the forward model, existence and continuity of the posterior measure, characterization of maximum a posteriori (MAP) estimates and state-of-the-art MCMC methodology. Because the forward MAP is non-linear, the unicity of posterior modes is unclear and it appears that local maximas are present. Nevertheless, the Bayesian methodology provides both a regularized solution to the problem, while giving a quantification of uncertainty. However, the estimation of prior hyper-parameters is still out of reach, giving poor confidence in the estimated variance. This direction requires further research, that we will try to address in a future work.

\section{Aknowledgments}

This work has been supported by Colciencias and ECOS-Nord under the grant C15M04. The authors would like to thank Dr. Xavier Bay for his precious advice and proof reading.

\appendix

\section{Proofs}

\begin{proof}[Proof of theorem \ref{ThmSolutionOperator}]
	Using standard notations in PDE theory, let $\Omega=]0,L[$, $H=L^2(\Omega)$, $V=H^1_0(\Omega)$, $V^*=H^{-1}(\Omega)$, $\mathcal{P}=\mathbb{R}\times]0,\infty[\times L^2([0,T],V^*)$ and
	\begin{equation*}
		W([0,T],H,V)=\left\lbrace y\in L^2([0,T],H),\;y'\in L^2([0,T],V^*)\right\rbrace,
	\end{equation*}
	equipped with the norm $\lVert y\rVert_{W([0,T],H,V)}=\left(\lVert y\rVert^2_{L^2([0,T],V)}+\lVert y'\rVert^2_{L^2([0,T],V^*)}\right)^{\frac{1}{2}}$,
	then the weak form associated to equation \ref{Eq_PDE} has the following reduced form:
	\begin{equation*}
	\text{Find }y\in W([0,T],H,V),\;\langle F(y,u),v\rangle=0,\;\forall v\in L^2([0,T],V),
	\end{equation*}
	with
	\begin{equation*}
	\begin{split}
	\langle F(y,u),v\rangle&=\int_{0}^{T}\langle y_t(t)-f(t),v(t)\rangle_{V^*,V}+\lambda\langle y(t),v(t)\rangle_{H}+D\langle y_x(t),v_x(t)\rangle_{H}dt,\\
	&=\int_{0}^{T}\langle y_t(t)-f(t),v(t)\rangle_{V^*,V}+B(y(t),v(t),\lambda,D)dt.
	\end{split}
	\end{equation*}
	\begin{itemize}
		\item \textbf{Existence of a solution map}. The PDE operator is uniformly parabolic whenever $D>0$, thus there exists a unique weak solution of equation \ref{Eq_PDE} in $W(0,T,H,V)$ for every $u\in\mathcal{P}$ (see chapter 7 in \cite{Evans2010} for instance). Moreover, we have the following energy estimate:
		\begin{equation*}
			\lVert y(u)\rVert_{W(0,T,H,V)}\leq \frac{C}{\sqrt{D}}\lVert f\rVert_{L^2([0,T],V^*)}
		\end{equation*}
		with $C$ a constant independent of $u$.
		\item \textbf{Second order Fréchet differentiability of the solution map}. Let $u,h^u\in\mathcal{P}$ such that  $u+h^u\in\mathcal{P}$, and $h^y\in W(0,T,H,V)$ then $\forall v\in L^2([0,T],V)$:
		\begin{equation*}
		\langle F(y+h^y,u+h^u)-F(y,u),v\rangle=\langle DF(y,u)[h^y,h^u],v\rangle+c(h^u,h^y)
		\end{equation*}
		with
		\begin{equation*}
		\vert c(h^u,h^y)\vert=\left\vert h^\lambda\int_{0}^{T}\langle h^y,v\rangle_H dt+h^D\int_0^T\langle h^y_x,v_x\rangle_H dt\right\vert\leq C\lVert v\rVert_{L^2([0,T],V)}\lVert h\rVert_U^2,
		\end{equation*}
		with $C$ an other constant independent from $u$ and:
		\begin{equation*}
		\begin{split}
		\langle DF(y,u)[h^y,h^u],v\rangle&=\int_{0}^{T}\langle h_t^y,v\rangle_{V^*,V}dt+\int_{0}^{T}h^\lambda\langle y,v\rangle_Hdt\\
		&+\int_{0}^{T}h^D\langle y_x,v_x\rangle_Hdt+\int_{0}^{T}\lambda\langle h^y,v\rangle_Hdt\\
		&+D\int_{0}^{T}\langle h^y_x,v_x\rangle_Vdt-\int_{0}^{T}\langle f,v\rangle_{V^*,V}dt.
		\end{split}
		\end{equation*}
		Moreover, we have:
		\begin{equation*}
		\vert \langle DF(y,u)[h^y,h^u],v\rangle\vert\leq\lVert (h^u,h^y)\rVert\lVert v\rVert
		\end{equation*}
		thus $DF$ is bounded, which shows that $F$ is Fréchet-differentiable on $\mathcal{P}$. Consider $F_y$ the partial derivative of $F$ w.r.t. its first variable:
		\begin{equation*}
		\langle F_y(y,u)h^y,v\rangle=\int_{0}^{T}\langle h_t^y,v\rangle+\lambda \langle h^y,v\rangle+D\langle h_x^y,v_x\rangle dt.
		\end{equation*}
		which defines a unique solution $h\in W(0,T,H,V)$ whenever $D>0$ (using same arguments than previously). Here, $F_y^{-1}$ is clearly bounded. Because $F$ is differentiable and $F_y^{-1}$ exists and is bounded, the implicit function theorem applies and $y$ is differentiable on $\mathcal{P}$. The second order differentiability uses the same arguments.
		\item \textbf{Local Lipschitz continuity of the solution map}. Let $u\in\mathcal{P}$, $r>0$ such that $\mathcal{B}(u,r)\subset\mathcal{P}$ and $(u_1,u_2)\in\mathcal{B}(u,r)$. There exists two unique solution with respect to $(u_1,u_2)$ satisfying $\forall i\in\lbrace 1,2\rbrace$, $\forall v\in L^2([0,T],V)$:
		\begin{equation*}
			\langle y_i',v\rangle_{V^*,V}+B(y_i,v,\lambda_i,D_i)=\langle f_i,v\rangle_{V^*,V}\;\text{for almost every }t\in[0,T],
		\end{equation*}
		which leads to $\forall v\in L^2([0,T],V)$ and almost-every $t\in[0,T]$:
		\begin{equation*}
			\langle y_1'-y_2',v\rangle_{V^*,V}+B(y_1,v,\lambda_1,D_1)-B(y_2,v,\lambda_2,D_2)=\langle f_1-f_2,v\rangle_{V^*,V}.
		\end{equation*}
		and equivalently:
		\begin{equation*}
			\begin{split}
			&\langle y_1'-y_2',v\rangle_{V^*,V}+B(y_1-y_2,v,\lambda_1,D_1)\\
			&=\langle f_1-f_2,v\rangle_{V^*,V}+B(y_2,v,\lambda_1,D_1)-B(y_2,v,\lambda_2,D_2).
			\end{split}
			\label{eq_estimationv}
		\end{equation*}
		Now, letting $v=y_1-y_2$ and using the coercivity of $B$ we get:
		\begin{equation}
			\begin{split}
			&\langle y_1'-y_2',y_1-y_2\rangle_{V^*,V}+D_1\lVert y_1-y_2\rVert^2_{V}\\
			&\leq\langle f_1-f_2,y_1-y_2\rangle_{V^*,V}+B(y_2,y_1-y_2,\lambda_1,D_1)-B(y_2,y_1-y_2,\lambda_2,D_2).
			\end{split}
			\label{eq_estimation}
		\end{equation}
		Now, we use the identity $\frac{d}{dt}\left(\frac{1}{2}\lVert y_1-y_2\rVert_H^2\right)=\langle y_1'-y_2',y_1-y_2\rangle_{V^*,V}$ and drop the positive term in the left hand side to get
		\begin{equation*}
		\begin{split}
		&\frac{d}{dt}\left(\frac{1}{2}\lVert y_1-y_2\rVert_H^2\right)\\
		&\leq \langle f_1-f_2,y_1-y_2\rangle_{V^*,V}+B(y_2,y_1-y_2,\lambda_1,D_1)-B(y_2,y_1-y_2,\lambda_2,D_2).
		\end{split}
		\end{equation*}
		We integrate between $0$ and $t$, use Cauchy-Schwartz and Poincarré's inequalities and $y_1(0)=y_2(0)=0$ to obtain the following estimation
		\begin{equation*}
			\frac{1}{2}\lVert y_1-y_2\rVert_H^2\leq C\lVert y_1-y_2\rVert_{L^2([0,T],V)}\lVert u_1-u_2\rVert,\;t-a.e.
		\end{equation*}
		The same reasoning applies now to equation \ref{eq_estimation} to get
		\begin{equation*}
			\lVert y_1-y_2\rVert_{L^2([0,T],V)}\leq C\lVert u_1-u_2\rVert.
		\end{equation*}
		We start back from equation \ref{eq_estimationv} to get:
		\begin{equation*}
			\begin{split}
			&\langle y_1'-y_2',v\rangle_{V^*,V}=-B(y_1-y_2,v,\lambda_1,D_1)\\
			&+\langle f_1-f_2,v\rangle_{V^*,V}+B(y_2,v,\lambda_1,D_1)-B(y_2,v,\lambda_2,D_2),
			\end{split}
		\end{equation*}
		and taking the supremum on the unit ball of $V$ before integrating:
		\begin{equation*}
			\begin{split}
			&\lVert y_1'-y_2'\rVert_{L^2([0,T],V^*)}\leq C_2\lVert u_1-u_2\rVert
			\end{split}
		\end{equation*}
		which completes the proof.
	\end{itemize}
\end{proof}

\begin{proof}[Proof of theorem \ref{ThmPosteriorMeasure}]
	Following the program in \cite{Dashti2015}, let us first show that $\mu_z$ exists and is unique and then the continuity of the posterior with respect to the Hellinger distance.
	\begin{itemize}
		\item (Existence and unicity of $ \mu_z$). Let $\mu_0$ be a probability measure defined as in equation \ref{EqPriorMeasure} and consider the following Gaussian negative log-likelihood:
		\begin{equation*}
		\Phi(u;z)=\frac{1}{2\sigma_\eta^2}\lVert z-\mathcal{G}(u)\rVert^2.
		\end{equation*}
		Since $\mathcal{G}$ is differentiable (theorem \ref{ThmSolutionOperator}), $\Phi$ is measurable w.r.t. $\mu_0$ for all $z\in\mathbb{R}^n$. Consider now the following integral
		\begin{equation*}
		Z(z)=\int_{\mathcal{U}}\exp(-\Phi(u;z))\mu_0(du).
		\end{equation*}
		The negative log-likelihood being positive, the integral is finite. Furthermore, we have $\forall (u,z)\in\mathcal{U}\times \mathbb{R}^n,\;\Phi(u;z)<+\infty$ thus $Z(z)>0$ for every $z\in\mathbb{R}^n$. In other words, the following function
		\begin{equation*}
		\frac{\exp(-\Phi(u;z))}{Z(z)}
		\end{equation*}
		defines a probability density function w.r.t. $\mu_0$, and the associated measure is $ \mu_z$, the (unique) posterior distribution of $u\vert z$.
		\item (Continuity in $z$). It remains to show that the posterior distribution is continuous in $z$ w.r.t. the Hellinger distance. Following the lines of \cite{Stuart2010a,Dashti2015} we proove the following sufficient condition
		\begin{equation*}
			\vert \Phi(u;z_1)-\Phi(u;z_2)\vert\leq g(u)\lVert z_1-z_2\rVert,
		\end{equation*}
		with $g\in L^1(\mu_0)$.
		\begin{eqnarray*}
			\vert \Phi(u;z_1)-\Phi(u;z_2)\vert&=&\frac{1}{2\sigma_\epsilon^2}\left\vert\lVert z_1\rVert^2-\lVert z_2\rVert^2+2\langle z_2-z_1,\mathcal{G}(u)\rangle\right\vert,\\
			&=&\frac{1}{2\sigma_\epsilon^2}\left\vert(\lVert z_1\rVert-\lVert z_2\rVert)(\lVert z_1\rVert+\lVert z_2\rVert)+2\langle z_2-z_1,\mathcal{G}(u)\rangle\right\vert,\\
			&\leq&\frac{1}{\sigma_\eta^2}(r+\lVert\mathcal{G}(u)\rVert)\lVert z_1-z_2\rVert.
		\end{eqnarray*}
		Now, using continuous injection between spaces, we have
		\begin{equation*}
			\lVert\mathcal{G}(u)\rVert\leq n\lVert y(u)\rVert_\infty\leq nC_1\lVert y(u)\rVert_{W}\leq nC_2\frac{\lVert f\rVert_{L^2}}{D},
		\end{equation*}
		the last inequality is given by the energy estimate of theorem \ref{ThmSolutionOperator}. It remains to see that
		\begin{equation*}
			g(u):=\frac{1}{\sigma_\eta^2}\left(r+\frac{\lVert f\rVert}{D}\right)
		\end{equation*}
		is in $L^1(\mu_0)$, that is
		\begin{equation*}
			\mathbb{E}^{\mu_0}\left[g(u)\right]=\mathbb{E}^{\mu_0^D}\left[\frac{1}{D}\right]\mathbb{E}^{\mu_0^f}[\lVert f\rVert_{L^2}].
		\end{equation*}
		The first term in the right hand side product is finite by assumption on $\mu_0^D$. The second term is finite as well since
		\begin{equation*}
			\lVert \exp(f^*)\rVert_{L^2}\leq\lVert \exp(f^*)\rVert_\infty\leq\exp(\lVert f^*\rVert_\infty)
		\end{equation*}
		which is $\mu_0^f$ integrable by the assumption on $\mu_0^f$. Finally, Fubini's theorem gives the result on the integrability of $g$. Now, it remains to follow the proof from \cite{Stuart2010a} as is.
	\end{itemize}
\end{proof}

\begin{proof}[Proof of theorem \ref{ThmMAP}]
	The posterior measure $\mu_z$ is absolutely continuous w.r.t. $\mu$ with the following Radon-Nikodym density:
	\begin{equation*}
		\frac{d \mu_z}{d\mu_{ref}}(u)=\frac{d \mu_z}{d\mu_0}(u)\frac{d\mu_0}{d\mu_{ref}}(u)=\frac{1}{Z(z)}\exp(-\tilde{\Phi}(u;z)).
	\end{equation*}
	where $\tilde{\Phi}(u;z)=\Phi(u;z)-\ln\left(\frac{d\mu_0}{d\mu_{ref}}(u)\right)$. Let us now show that $\Phi$ is locally Lispchitz in its first argument:
	\begin{eqnarray*}
		\vert\Phi(u_1;z)-\Phi(u_2;z)\vert&=&\frac{1}{2\sigma_\eta^2}\left\vert \lVert z-\mathcal{G}(u_1)\rVert^2-\lVert z-\mathcal{G}(u_2)\rVert^2\right\vert,\\
		&=&\frac{1}{2\sigma_\eta^2}\left\vert \lVert \mathcal{G}(u_1)\rVert^2-\lVert \mathcal{G}(u_2)\rVert^2+2\langle z,\mathcal{G}(u_2)-\mathcal{G}(u_1)\rangle\right\vert,\\
		&\leq& \frac{\lVert \mathcal{G}(u_1)\rVert+\lVert \mathcal{G}(u_2)\rVert+2\lVert z\rVert}{2\sigma_\eta^2}\lVert \mathcal{G}(u_1)-\mathcal{G}(u_2)\rVert,
	\end{eqnarray*}
	and since $\lVert \mathcal{G}(u_1)-\mathcal{G}(u_2)\rVert\leq C\lVert y(u_1)-y(u_2)\rVert$, we conclude that $\Phi$ and
	thus $\tilde{\Phi}$ are locally Lipschitz. Now, we follow the lines of \cite{Dashti2013} it comes:
	\begin{eqnarray*}
		\frac{J^\delta(u_1)}{J^\delta(u_2)}&=&\frac{\int_{B^\delta(u_1)}\exp(-\tilde{\Phi}(u;z))\mu_{ref}(du)}{\int_{B^\delta(u_2)}\exp(-\tilde{\Phi}(v;z))\mu_{ref}(dv)},\\
		&=&\frac{\int_{B^\delta(u_1)}\exp(-\tilde{\Phi}(u;z)+\tilde{\Phi}(u_1;z))\exp(-\tilde{\Phi}(u_1;z))\mu_{ref}(du)}{\int_{B^\delta(u_2)}\exp(-\tilde{\Phi}(v;z)+\tilde{\Phi}(u_2;z))\exp(-\tilde{\Phi}(u_2;z))\mu_{ref}(dv)}.
	\end{eqnarray*}
	Now,
	\begin{eqnarray*}
		\frac{J^\delta(u_1)}{J^\delta(u_2)}\leq\exp\left(\delta C-\tilde{\Phi}(u_1,z)+\tilde{\Phi}(u_2;z)\right)\frac{\int_{B^\delta(u_1)}\mu_{ref}(du)}{\int_{B^\delta(u_2)}\mu_{ref}(dv)}
	\end{eqnarray*}
	and finally
	\begin{equation*}
		\limsup_{\delta\to 0}\frac{J^\delta(u_1)}{J^\delta(u_2)}\leq\exp\left(-I(u_1)+I(u_2)\right).
	\end{equation*}
	A similar argument leads to
	\begin{equation*}
	\liminf_{\delta\to 0}\frac{J^\delta(u_1)}{J^\delta(u_2)}\geq\exp\left(-I(u_1)+I(u_2)\right).
	\end{equation*}
	We conclude that $\lim_{\delta\to 0}\frac{J^\delta(u_1)}{J^\delta(u_2)}=\exp\left(-I(u_1)+I(u_2)\right)$. For a fixed value $u_2$, this quantity is maximized when $u_1$ is a minimizer of $I$, the proof is then complete.
\end{proof}

\bibliographystyle{plain}
\bibliography{Bib_Melanogaster}

\end{document}